\documentclass[12pt, a4]{amsart}

\usepackage{latexsym,amsmath,amssymb,amsthm,mathrsfs,color}
\usepackage{float}
\usepackage{graphicx}
\usepackage{marginnote}
\usepackage{scalerel}
\usepackage{stackengine,wasysym}

\usepackage[bbgreekl]{mathbbol}
\usepackage{bbm}
\usepackage{tikz-cd}
\usetikzlibrary{graphs,decorations.pathmorphing,decorations.markings}

\usepackage[all]{xy}
\usepackage{stmaryrd}

\usepackage{hyperref}
\hypersetup{colorlinks,%
    citecolor=brown,%
    filecolor=black,%
   linkcolor=blue,%
   urlcolor=black}

\usepackage[toc,page]{appendix}

\usepackage[shortalphabetic]{amsrefs}
 \makeatletter
\def\append@label@year@{%
    \safe@set\@tempcnta\bib@year
    \edef\bib@citeyear{\the\@tempcnta}%
    \ifnum\bib@citeyear>9
      \append@to@stem{%
          \ifx\bib@year\@empty
          \else
            \@xp\year@short \bib@citeyear \@nil
          \fi
      }%
    \fi
}
\makeatother

\let\oldtocsection=\tocsection
\renewcommand{\tocsection}[2]{\hspace{0em}\oldtocsection{#1}{#2}}

\usepackage[a4paper]{geometry}

\def\upddots{\mathinner{\mkern 1mu\raise 1pt \hbox{.}\mkern 2mu
\mkern 2mu \raise 4pt\hbox{.}\mkern 1mu \raise 7pt\vbox {\kern 7
pt\hbox{.}}} }

\numberwithin{equation}{section}

\begin{document}
\setlength{\unitlength}{2.5cm}

\newtheorem{thm}{Theorem}[section]
\newtheorem{lm}[thm]{Lemma}
\newtheorem{prop}[thm]{Proposition}
\newtheorem{cor}[thm]{Corollary}
\newtheorem{conj}[thm]{Conjecture}
\newtheorem{specu}[thm]{Speculation}

\theoremstyle{definition}
\newtheorem{dfn}[thm]{Definition}
\newtheorem{eg}[thm]{Example}
\newtheorem{rmk}[thm]{Remark}

\newcommand{\F}{\mathbf{F}}
\newcommand{\N}{\mathbbm{N}}
\newcommand{\R}{\mathbbm{R}}
\newcommand{\C}{\mathbbm{C}}
\newcommand{\Z}{\mathbbm{Z}}
\newcommand{\Q}{\mathbbm{Q}}
\newcommand{\Mp}{{\rm Mp}}
\newcommand{\Sp}{{\rm Sp}}
\newcommand{\GSp}{{\rm GSp}}
\newcommand{\GL}{{\rm GL}}
\newcommand{\PGL}{{\rm PGL}}
\newcommand{\SL}{{\rm SL}}
\newcommand{\SO}{{\rm SO}}
\newcommand{\Spin}{{\rm Spin}}
\newcommand{\GSpin}{{\rm GSpin}}
\newcommand{\Ind}{{\rm Ind}}
\newcommand{\Res}{{\rm Res}}
\newcommand{\Hom}{{\rm Hom}}
\newcommand{\End}{{\rm End}}
\newcommand{\msc}[1]{\mathscr{#1}}
\newcommand{\mfr}[1]{\mathfrak{#1}}
\newcommand{\mca}[1]{\mathcal{#1}}
\newcommand{\mbf}[1]{{\bf #1}}
\newcommand{\mbm}[1]{\mathbbm{#1}}
\newcommand{\into}{\hookrightarrow}
\newcommand{\onto}{\twoheadrightarrow}
\newcommand{\s}{\mathbf{s}}
\newcommand{\cc}{\mathbf{c}}
\newcommand{\bfa}{\mathbf{a}}
\newcommand{\id}{{\rm id}}
\newcommand{\g}{ \mathbf{g} }
\newcommand{\w}{\mathbbm{w}}
\newcommand{\Ftn}{{\sf Ftn}}
\newcommand{\p}{\mathbf{p}}
\newcommand{\bq}{\mathbf{q}}
\newcommand{\WD}{\text{WD}}
\newcommand{\W}{\text{W}}
\newcommand{\Wh}{{\rm Wh}}
\newcommand{\Whc}{{{\rm Wh}_\psi}}
\newcommand{\ggma}{\omega}
\newcommand{\sct}{\text{\rm sc}}
\newcommand{\Of}{\mca{O}^\digamma}
\newcommand{\gk}{c_{\sf gk}}
\newcommand{\Irr}{ {\rm Irr} }
\newcommand{\Irrg}{ {\rm Irr}_{\rm gen} }
\newcommand{\diag}{{\rm diag}}
\newcommand{\uchi}{ \underline{\chi} }
\newcommand{\Tr}{ {\rm Tr} }
\newcommand{\der}\de
\newcommand{\Stab}{{\rm Stab}}
\newcommand{\Ker}{{\rm Ker}}
\newcommand{\bfp}{\mathbf{p}}
\newcommand{\bfq}{\mathbf{q}}
\newcommand{\KP}{{\rm KP}}
\newcommand{\Sav}{{\rm Sav}}
\newcommand{\de}{{\rm der}}
\newcommand{\tnu}{{\tilde{\nu}}}
\newcommand{\lest}{\leqslant}
\newcommand{\gest}{\geqslant}
\newcommand{\tu}{\widetilde}
\newcommand{\tchi}{\tilde{\chi}}
\newcommand{\tomega}{\tilde{\omega}}
\newcommand{\Rep}{{\rm Rep}}
\newcommand{\cu}[1]{\textsc{\underline{#1}}}
\newcommand{\set}[1]{\left\{#1\right\}}
\newcommand{\ul}[1]{\underline{#1}}
\newcommand{\wt}[1]{\widetilde{#1} }
\newcommand{\wtsf}[1]{\wt{\sf #1}}
\newcommand{\anga}[1]{{\left\langle #1 \right\rangle}}
\newcommand{\angb}[2]{{\left\langle #1, #2 \right\rangle}}
\newcommand{\wm}[1]{\wt{\mbf{#1}}}
\newcommand{\elt}[1]{\pmb{\big[} #1\pmb{\big]} }
\newcommand{\ceil}[1]{\left\lceil #1 \right\rceil}
\newcommand{\floor}[1]{\left\lfloor #1 \right\rfloor}
\newcommand{\val}[1]{\left| #1 \right|}
\newcommand{\aff}{ {\rm aff} }
\newcommand{\ex}{ {\rm ex} }
\newcommand{\exc}{ {\rm exc} }
\newcommand{\HH}{ \mca{H} }
\newcommand{\HKP}{ {\rm HKP} }
\newcommand{\std}{ {\rm std} }
\newcommand{\motimes}{\text{\raisebox{0.25ex}{\scalebox{0.8}{$\bigotimes$}}}}

\newcommand{\FF}{\mca{F}}
\newcommand{\tv}{\tilde{v}}
\newcommand{\betaa}{\pmb{\beta}}
\newcommand{\deltaa}{\pmb{\delta}}
\newcommand{\bepsilon}{\bar{\epsilon}}

\title{Formal degrees of genuine Iwahori-spherical representations}

\author{Ping Dong, Fan Gao, and Runze Wang}
\address{School of Mathematical Sciences, Zhejiang University, 866 Yuhangtang Road, Hangzhou, China 310058}
\email{{pdmath@zju.edu.cn ({\rm P. D.}),  gaofan@zju.edu.cn ({\rm F. G.}),  wang\underline{\ }runze@zju.edu.cn ({\rm R. W.})}}

\date{}
\maketitle

\begin{abstract} 
For irreducible genuine Iwahori-spherical discrete series representations of central covers, we verify the Hiraga--Ichino--Ikeda formula for their formal degrees. We also compute the Whittaker dimensions of these representations, when their associated modules over the genuine Iwahori--Hecke algebra are one-dimensional.
\end{abstract}

\tableofcontents


\section{Introduction} \label{S:intro}

Let $F$ be a finite extension of $\Q_p$ with discrete valuation $|\cdot|_F$. Let $O_F \subset F$ be the ring of algebraic integers and let $\mfr{p}_F \subset O_F$ be its prime ideal. Fix a uniformizer $\varpi \in \mfr{p}_F$ with $\val{\varpi}_F=q^{-1}$. Let $\mbf{G}$ be  a connected semisimple algebraic group scheme over $O_F$ such that the geometric fiber is \emph{split} over $F$. Write $G:=\mbf{G}(F)$. Assume that $F^\times$ contains the full group $\pmb{\delta}_n$ of $n$-th roots of unity. We are interested in the $\epsilon$-genuine representations of the Brylinski--Deligne central cover $\wt{G}$ as in
$$\begin{tikzcd}
\pmb{\delta}_n \ar[r, hook] & \wt{G} \ar[r, two heads, "{\mfr{q}}"] & G,
\end{tikzcd}$$
where $\pmb{\delta}_n$ always acts via a fixed embedding $\epsilon: \pmb{\delta}_n \into \C^\times$. Denote by $\Irr_\epsilon(\wt{G})$ the set of isomorphism classes of $\epsilon$-genuine irreducible representations of $\wt{G}$.
For every subset $H \subset G$, we write $\wt{H}:=\mfr{q}^{-1}(H)$.

For an irreducible discrete series $\pi$ of the linear group $G$, it was conjectured by Hiraga--Ichino--Ikeda \cite{HII, HIIc} that (with respect to a canonically defined Haar measure $\mu_{G, \psi}^\natural$ depending on an additive character $\psi$ of $F$)
\begin{equation} \label{HII}
\deg(\pi, \mu_{G, \psi}^\natural) = \frac{\dim \tau_\pi}{ \val{\mca{S}_{\phi_\pi}} } \cdot \val{\gamma(0, \phi_\pi, {\rm Ad}, \psi)},
\end{equation} 
where $(\phi_\pi, \tau_\pi)$ is an enhanced Langlands parameter of $\pi$, and $\gamma(s, \phi_\pi, {\rm Ad}, \psi)$ the adjoint gamma-factor. We refer to \cite{HII, CKK12, Qiu12, GaIc14, CiOp15, ILM17, FeOp20, FOS20, Rap21, FOS22} and references therein for some recent work on \eqref{HII} in the past decade. We should remark also that prior to \eqref{HII}, there were already explicit computations of some formal degrees such as in the works of Borel \cite{Bor76}, Reeder \cite{Ree94, Ree-00} and DeBacker--Reeder \cite{DR}.

Assume that $p\nmid n$ and thus fix a splitting of $\wt{G}$ over $K:=\mbf{G}(O_F)$. This gives a splitting of $\wt{G}$ over the standard Iwahori subgroup $I \subset K$. Let $\HH_{\bepsilon}(\wt{G}, I)$ be the genuine Iwahori--Hecke algebra of $\wt{G}$ with respect to $I$. We consider the formal degree problem of genuine Iwahori-spherical representations of $\wt{G}$ and verify the following, which is probably well-known to the experts.

\begin{thm}[{Theorem \ref{T:fd-HII}}] \label{0T:main}
Let $\wt{G}$ be an $n$-fold cover of a split semisimple $G$ with $p\nmid n$. Let $\pi_\sigma \in \Irr_\epsilon(\wt{G})$ be the genuine Iwahori-spherical representation associated with $\sigma \in \Irr(\HH_{\bepsilon}(\wt{G}, I))$. Let $\mca{L}_{\wt{G}}^{\rm DL}$ be the Deligne--Langlands correspondence for $\wt{G}$ given in Definition \ref{DLcor}, depending on the choice of the Deligne--Langlands correspondence for split linear groups as in \cite[Theorem 3.1]{FOS22}.
Let $(\phi_\sigma, \tau_\sigma):=\mca{L}_{\wt{G}}^{\rm DL}(\pi_\sigma)$ be the Deligne--Langlands parameter of $\pi_\sigma$.  Assume that $\psi$ is of conductor $O_F$ and $\pi_\sigma$ is a discrete series representation. Then one has
$$\deg(\pi_\sigma, \mu_{\wt{G}, \psi}^\natural) = \frac{\dim \tau_\sigma}{ \val{ \mca{S}_{\phi_\sigma} } } \cdot \val{\gamma(0, \phi_\sigma, {\rm Ad}, \psi)},$$
where $\mu_{\wt{G}, \psi}^\natural$ is the measure of $\wt{G}$ naturally associated with the canonical measure $\mu_{G, \psi}^\natural$, as given in \eqref{E:meas}.
\end{thm}

We briefly summarize the main steps in the proof of Theorem \ref{0T:main}, which occupy \S \ref{S:par} and \S \ref{S:fd}:
\begin{enumerate}
\item[$\bullet$]  In \S \ref{S:par}, we define a Deligne--Langlands correspondence for irreducible genuine Iwahori-spherical representations $\pi$ of $\wt{G}$, i.e., we associate an enhanced parameter $(\phi_\pi, \tau_\pi)$ to such $\pi$, see Definition \ref{DLcor}. This is achieved in several steps.
First, the map $\pi \mapsto \pi^I$ gives a bijection between irreducible  genuine Iwahori-spherical representations of $\wt{G}$ and irreducible $\HH_{\bepsilon}(\wt{G}, I)$-modules. Denote this correspondence  by 
$$\pi_\sigma \leftrightarrow \sigma,$$
where $\sigma \in \Irr(\HH_{\bepsilon}(\wt{G}, I))$. Second, by using a distinguished genuine character $\theta$, one has an algebra isomorphism $\varphi_\theta: \HH_{\bepsilon}(\wt{G}, I) \simeq \HH(G_{Q,n}, I_{Q,n})$, where $G_{Q,n}$ is the linear algebraic group satisfying $G_{Q,n}^\vee \simeq \wt{G}^\vee$. By transport of action, one has an irreducible module $\sigma_\theta:=\varphi_\theta^*(\sigma)$ over $\HH(G_{Q,n}, I_{Q,n})$.
Then, using a Kazhdan--Lusztig--Reeder (KLR) parametrization for $\sigma_\theta$, one obtains an enhanced  $L$-parameter $(\phi_{\sigma_\theta}, \tau_{\sigma_\theta})$ with $\phi_{\sigma_\theta}$ valued in $G_{Q,n}^\vee \simeq \wt{G}^\vee$. Lastly, the distinguished genuine character $\theta$ also gives an injection $\wt{G}^\vee \into {}^L\wt{G}$, and thus one has an $L$-parameter $\phi_\pi$ of $\pi$ valued in ${}^L\wt{G}$. Similar consideration applies in order to define $\tau_\pi$ using $\tau_{\sigma_\theta}$, and this gives the sought Deligne--Langlands parameter $(\phi_\pi, \tau_\pi)$ of $\pi$. The parameter $(\phi_\pi, \tau_\pi)$ depends on a choice of the KLR parametrization for $G_{Q,n}$, but does not depend  on the choice of $\theta$. This latter independence on $\theta$ is shown in Proposition \ref{P:indep}.
\item[$\bullet$]  In \S \ref{S:fd}, we start with reducing the computation of $\deg(\pi_\sigma, \mu_{\wt{G}})^{-1}$ to a sum over the modified extended Weyl group $\wt{W}_{\rm ex}$, see Lemma \ref{L:2fdeg}. Here the measure $\mu_{\wt{G}}$ is such that $\mu_{\wt{G}}(\pmb{\delta}_n \times I) =1$. In fact, the right hand side of the equality in Lemma \ref{L:2fdeg}, in the linear algebraic case, is often taken as the definition of the inverse of the formal degree for discrete series representations of Iwahori--Hecke algebras. Let $\pi_{\sigma_\theta} \in \Irr(G_{Q,n})$ be the Iwahori-spherical representation associated with $\sigma_\theta$ given as above. In Proposition \ref{P:key-HII} we make a term by term comparison between $\deg(\pi_\sigma, \mu_{\wt{G}})^{-1}$ and $\deg(\pi_{\sigma_\theta}, \mu_{G_{Q,n}})^{-1}$.
By a change of Haar measures and invoking the fact that Hiraga--Ichino--Ikega formula for $\pi_{\sigma_\theta}$ is verified in \cite[Theorem 2]{FOS22} for the linear group $G_{Q,n}$ (with respect to the KLR parametrization for $\sigma_\theta$ given in \cite[Theorem 3.1]{FOS22} there), we immediately obtain Theorem \ref{0T:main}.
\end{enumerate}

In the last section \S \ref{S:Wdim}, we fix $G$ to be almost simple and simply-connected.  We compute for the $n$-fold ``oasitic covers" $\wt{G}$ of such $G$ the Whittaker dimension of the discrete series $\pi_\sigma$, when $\sigma$ is one-dimensional. The Whittaker dimension is a polynomial of $n$, as $n$ varies. This resonates with the fact that the wavefront sets of  genuine representations of a cover tend to be larger. Moreover, the wavefront sets of $\pi_\sigma$ are not dictated by their $L$-parameters, and this is in contrast with their formal degrees as shown in Theorem \ref{0T:main} above.


\subsection{Acknowledgement} We would like to thank Yongqi Feng, Atsushi Ichino and Jiandi Zou for very helpful communications on the topics discussed here or on the content of the paper. The work of F. G. is partially supported by the National Key R\&D Program of China (No. 2022YFA1005300) and also by NNSFC-12171422.

\section{Covers and generalities on discrete series}

\subsection{Central covers} \label{SS:cov}
We retain the notation from the beginning of \S \ref{S:intro}. Let $G$ be the $F$-rational points of a semisimple group scheme $\mbf{G}$ defined over $O_F$. We assume that the geometric fiber (still denoted by) $\mbf{G}$ is split over $F$. Denote by 
$$(X,\ \Phi, \ \Delta; \ Y, \Phi^\vee, \Delta^\vee)$$
 the root datum of $G$, where $X$ is the character lattice and $Y$ the cocharacter lattice of a split torus $T\subset G$. Here $\Delta$ is a choice of simple roots, and $Y^{sc} \subset Y$ is the coroot sublattice and $X^{sc} \subset X$ the root lattice. Denote by $W$ the Weyl group of the coroot system.

Let $$Q: Y \longrightarrow \Z$$ be a $W$-invariant quadratic form, and let $B_Q(y, z):=Q(y+z) - Q(y) - Q(z)$ be the associated bilinear form. Assume that $F^\times$ contains the full group $\pmb{\delta}_n$ of $n$-th roots of unity.  For any pair $(D, \eta=\mbf{1})$ with $D$ being a ``bisector" of $Q$ (see \cite[\S 2.6]{GG}), one has an associated Brylinski--Deligne covering group  $\wt{G}_D^{(n)}$, which is a central extension
$$\begin{tikzcd}
\pmb{\delta}_n \ar[r, hook] & \wt{G}_D^{(n)} \ar[r, two heads, "{\mfr{q}}"] & G
\end{tikzcd}$$
of $G$ by $\pmb{\delta}_n$, for details see \cite{BD, GG, We6}. For simplicity, we will write $\wt{G}^{(n)}$ or even just $\wt{G}$ for $\wt{G}_D^{(n)}$. For every subset $H \subset G$ we denote $\wt{H}:=\mfr{q}^{-1}(H)$. For a fixed embedding
$$\epsilon: \pmb{\delta}_n \into \C^\times,$$
an $\epsilon$-genuine representation is one such that $\pmb{\delta}_n$ acts via $\epsilon$.

The complex dual group $\wt{G}^\vee$ of $\wt{G}$ is determined by its root datum
$$(Y_{Q,n}, \ \Delta_{Q,n}^\vee;\ X_{Q,n},\ \Delta_{Q,n}).$$
Here $Y_{Q,n} \subset Y$ is the character lattice of $\wt{G}^\vee$ given by 
$$Y_{Q,n}=\set{y\in Y: \ B_Q(y, z)\in n\Z \text{ for all } z \in Y},$$
and $X_{Q,n}:=\Hom_{\Z}(Y_{Q,n}, \Z)$. The set $\Delta_{Q,n}^\vee$ consists of the re-scaled simple coroots 
$$\alpha_{Q,n}^\vee:=n_\alpha \alpha^\vee:=\frac{n}{\gcd(n,Q(\alpha^\vee))} \alpha^\vee.$$
Let $Y_{Q,n}^{sc} \subset Y_{Q,n}$ be the sublattice spanned by $\Delta_{Q,n}^\vee$. 
We have
$$Z(\wt{G}^\vee) = \Hom(Y_{Q,n}/Y_{Q,n}^{sc}, \C^\times).$$
We also set $\alpha_{Q,n}:=n_\alpha^{-1}\alpha$ for every root $\alpha$, and thus $\Delta_{Q,n}:=\set{\alpha_{Q,n}}_{\alpha \in \Delta}$.

Let $\WD_F = W_F\times \SL_2(\C)$ be the Weil--Deligne group of $F$. In \cite{We6}, Weissman constructed the L-group ${}^L\wt{G}$ for $\wt{G}$, which is an extension
$$\begin{tikzcd}
\wt{G}^\vee \ar[r, hook] & {}^L\wt{G} \ar[r, two heads] & \WD_F.
\end{tikzcd}$$
Here ${}^L\wt{G}$ arises from an abelian central extension 
$$\begin{tikzcd}
Z(\wt{G}^\vee) \ar[r, hook] & E_{\wt{G}} \ar[r, two heads] & F^\times
\end{tikzcd}$$
via the push-out by the inclusion $Z(\wt{G}^\vee) \into \wt{G}^\vee$ and the pull-back via the reciprocity map $\WD_F \onto W_F \to F^\times$. One always has ${}^L\wt{G} \simeq \wt{G}^\vee \rtimes \WD_F$; equivalently, the extension of ${}^L\wt{G}$ splits, where one of such splittings is provided by the local Langlands correspondence for covering torus. However, in general it is not true that ${}^L\wt{G} \simeq \wt{G}^\vee \times \WD_F$. If the extension $E_{\wt{G}}$ splits over $F^\times$, then ${}^L\wt{G}$ is isomorphic to $\wt{G}^\vee \times \WD_F$; this is trivially true if $Z(\wt{G}^\vee)=\set{1}$, i.e., if $\wt{G}^\vee$ is of adjoint type.

\subsection{Distinguished genuine characters} \label{SS:dis}
Let $\wt{T} \subset \wt{G}$ be the covering torus of $\wt{G}$. We assume that there exists a ``distinguished" genuine central character $\theta: Z(\wt{T}) \to \C^\times$, first introduced by Savin \cite{Sav04}. See \cite[\S6]{GG} for a discussion on the necessary conditions for its existence. In particular, it is also shown in \cite[Theorem 6.6]{GG} that the set of all distinguished central characters of $Z(\wt{T})$, whenever exist, is a torsor over
$$\Hom(F^\times/F^{\times n}, Z^\heartsuit(\wt{G}^\vee)),$$
where 
$$Z^\heartsuit(\wt{G}^\vee):=\Hom(Y_{Q,n}/(nY + Y_{Q,n}^{sc}), \C^\times) \subset Z(\wt{G}^\vee).$$
Note that if $G$ is simply-connected and semisimple, then $nY \subset Y_{Q,n}^{sc}$ and thus $Z^\heartsuit(\wt{G}^\vee) = Z(\wt{G}^\vee)$.

Among several good properties of a distinguished genuine character $\theta$ are the following:
\begin{enumerate}
\item[(i)] It is always Weyl-invariant and satisfies
$$\theta(\wt{h}_\alpha(a^{n_\alpha})) =1$$
for every $a\in F^\times$ and $\alpha\in \Delta$. Here $\wt{h}_\alpha(a^{n_\alpha}) \in Z(\wt{T})$ is a natural lifting of the torus element $h_\alpha(a^{n_\alpha}) \in T$, defined using the Steinberg relations.
\item[(ii)] Every such $\theta$ gives rise to a splitting of $E_{\wt{G}}$ over $F^\times$ and thus an isomorphism
\begin{equation} \label{D:LGiso}
i_\theta: {}^L\wt{G} \simeq_\theta \wt{G}^\vee \times \WD_F.
\end{equation}
\end{enumerate} 
Relying on a nontrivial additive character $\psi: F\to \C^\times$ and thus the associated Weil index, a detailed construction of distinguished genuine characters is given in \cite[\S 7]{GG}. 

\subsection{Discrete series and formal degrees} \label{SS:ds-fd}
For an embedding $\epsilon: \pmb{\delta}_n \into \C^\times$, we denote by $\Irr_\epsilon(\wt{G})$
the set of isomorphism classes of irreducible $\epsilon$-genuine representations of $\wt{G}$. 
For every $(\pi, V) \in \Irr_\epsilon(\wt{G})$, we have the contragredient representation $(\pi^\vee, V^\vee)$ which is realized by the smooth vectors in the algebraic dual of $V$. The action is specified by requiring
$$\angb{v}{\pi^\vee(g)(\lambda)} =\angb{\pi(g^{-1}) v}{\lambda}$$
for every  $v\in V$ and $\lambda \in V^\vee$, where $\angb{-}{-}: V \times V^\vee \to \C$ is the canonical pairing. In particular, we have
$$\pi^\vee \in \Irr_{\bepsilon}(\wt{G}),$$
where $\bepsilon=\epsilon^{-1}$ is the inverse of $\epsilon$.

For every pair $(v, \lambda) \in V \times V^\vee$, we have the function of matrix coefficient 
$$f_{v, \lambda}: \wt{G} \longrightarrow \C$$
given by $f_{v, \lambda}(g): =\angb{\pi(g^{-1}) v}{\lambda}$. 
Let $\mu_{\wt{G}}$ be a Haar measure of $\wt{G}$. A representation $\pi \in \Irr_\epsilon(\wt{G})$ is called a discrete series  if every matrix coefficient $f_{v, \lambda}$ of $\pi$ satisfies
$$\int_{\wt{G}} \val{f_{v, \lambda}(g)}^2 \mu_{\wt{G}}(g)  < \infty.$$
Let $\pi \in \Irr_\epsilon(\wt{G})$ be a discrete series representation. Then there exists $\deg(\pi, \mu_{\wt{G}}) \in \R_{>0}$ such that
\begin{equation} \label{E:fd} 
\int_{\wt{G}} f_{v_1, \lambda_1}(g) \cdot f_{v_2, \lambda_2}(g^{-1})\ \mu_{\wt{G}}(g) = \frac{ \angb{v_2}{\lambda_1}\cdot \angb{v_1}{\lambda_2}  }{ \deg(\pi, \mu_{\wt{G}}) }
\end{equation}
for all $v_i, \lambda_i$, see \cite[Chapitre IV3.3]{Ren10} or \cite[Page 5]{HC70}. In this case, the number $\deg(\pi, \mu_{\wt{G}})$ is called the formal degree of $\pi$ with respect to the measure $\mu_{\wt{G}}$. Note that for any $c\in \C^\times$,
$$\deg(\pi, c\cdot \mu_{\wt{G}}) = c^{-1} \cdot \deg(\pi, \mu_{\wt{G}}).$$

\section{Parameters for genuine Iwahori-spherical representations} \label{S:par}
The goal of this section is to define a Deligne--Langlands correspondence for irreducible genuine Iwahori-spherical representations of $\wt{G}$, by using a correspondence for linear algebraic groups already established by Kazhdan--Lusztig \cite{KL2} and Reeder \cite{Ree4}.
\subsection{Structure of $\mca{H}_{\bepsilon}(\wt{G}, I)$} \label{SS:HH}
Henceforth, we assume that $p\nmid n$ and fix a splitting $s_K$ of $\wt{G}$ over $K:=\mbf{G}(O_F)$. Let $I\subset K$ be the Iwahori subgroup associated with $\mbf{B}(\kappa_F)$, where $\kappa_F$ is the residue field of $F$. We view $K, I \subset \wt{G}$ via $s_K$.

Denote by $\mca{H}_{\bepsilon}(\wt{G}, I):=C_{\bepsilon, c}^\infty(I\backslash\wt{G}/I)$ the genuine Iwahori--Hecke algebra consisting of $\bepsilon$-genuine locally constant compactly supported functions $f$ on $\wt{G}$ which are $I$-biinvariant, i.e., $f$ satisfies
$$f(\zeta a g b) = \bepsilon(\zeta) \cdot f(g)$$
for all $\zeta \in \pmb{\delta}_n$ and $a, b\in I$. Throughout the paper, we fix the Haar measure $\mu_{\wt{G}}$ on $\wt{G}$ such that
$$\mu_{\wt{G}}(\pmb{\delta}_n \times I) =1.$$
The algebra multiplication on $\HH_{\bepsilon}(\wt{G}, I)$ is given by the convolution $f_1 * f_2$.

To recall the structure of $\HH_{\bepsilon}(\wt{G}, I)$, we consider
$$N_{Q,n}:= N_{\wt{G}}(\mbf{T}(O_F)),$$
the normalizer of $\mbf{T}(O_F)$ inside $\wt{G}$. 
The group $N_{Q,n}$ sits in a short exact sequence
\begin{equation}
\begin{tikzcd}
\pmb{\delta}_n \times \mbf{T}(O_F) \ar[r, hook] & N_{Q,n} \ar[r, two heads, "\wp"] & \wt{W}_{\rm ex}.
\end{tikzcd}
\end{equation}
where 
$$\wt{W}_{\rm ex}:=Y_{Q,n} \rtimes W$$
 is the modified extended affine Weyl group. It is known that the support of $\HH_{\bepsilon}(\wt{G}, I)$ is $I N_{Q,n} I$, see \cite{Sav04}.

Let $\theta  \in \Irr_\epsilon(Z(\wt{T}))$ be a Weyl-invariant distinguished genuine character as discussed in \S \ref{SS:dis}. In the tame setting when $p\nmid n$, we may further assume that $\theta$ is unramified, i.e., trivial on $Z(\wt{T}) \cap \mbf{T}(O_F)$. Following Savin \cite[Page 125]{Sav04}, to every $w\in \wt{W}_{\rm ex}$ one can naturally associate an element 
$$e_w^\theta \in \HH_{\bepsilon}(\wt{G}, I).$$
We recall the definition of $e_w^\theta$. First, since $\theta$ is $W$-invariant and trivial on the intersection of $\mbf{T}(O_F)$ and the subgroup of $K$ generated by $\set{\wt{w}_\alpha(-1): \alpha \in \Delta}$, we can extend $\theta$ to get
$$\theta: N_{Q,n} \longrightarrow \C^\times$$
such that $\theta( \wt{w}_\alpha(1))=1$ for every $\alpha \in \Delta$. Here $\wt{w}_\alpha(t) = \wt{e}_\alpha(t) \wt{e}_{-\alpha}(t^{-1}) \wt{e}_\alpha(t) \in \wt{G}$ is the element defined by using the canonical unipotent lifting $\wt{e}_\alpha(\cdot)$ of $e_\alpha(\cdot)$. For every $w\in \wt{W}_{\rm ex}$, let $\wt{w} := \wt{w}_{\alpha_1}(1) ... \wt{w}_{\alpha_k}(1) \in N_{Q,n}$ be the natural representative of $w=w_{\alpha_1} ... w_{\alpha_k}$ written in a minimal decomposition. 
Then one has a unique element $e_w^\theta \in \HH_{\bepsilon}(\wt{G}, I)$ supported on $\pmb{\delta}_n I \wt{w} I$ satisfying
\begin{equation} \label{e_w}
e_w^\theta(x) = 
\begin{cases}
\theta^{-1}(x) & \text{ if } x \in \wp^{-1}(w), \\
0 & \text{ if }  x \in N_{Q,n} - \wp^{-1}(w).
\end{cases}
\end{equation}

Let $W_{\rm ex}:=Y \rtimes W$ be the extended affine Weyl group associated with $G$. Recall the length function 
$$l_G: W_{\rm ex} \longrightarrow \N$$
that is explicitly given by 
\begin{equation} \label{l_G}
l_G(w)=\sum_{\alpha \in \Phi^+(s)} \val{\angb{\alpha}{y}} + \sum_{\alpha \in \Phi^-(s)} \val{\angb{\alpha}{y}+ 1},
\end{equation}
where $w=(y,s) \in W_{\rm ex}$ with $y\in Y, s\in W$ and $\Phi^+(s)$ (resp. $\Phi^-(s)$) consists of all positive roots such that $s^{-1}(\alpha)>0$ (resp. $s^{-1}(\alpha)<0$). For every $w_1, w_2\in \wt{W}_{\rm ex}$, we have
\begin{equation} \label{fac-e}
e_{w_1}^\theta * e_{w_2}^\theta = e_{w_1 w_2}^\theta
\end{equation}
whenever $l_G(w_1 w_2) = l_G(w_1) + l_G(w_2)$.

An element $y\in Y$ is called positive or dominant if $\angb{y}{\alpha} \gest 0$ for every simple root $\alpha$. The element $e_y^\theta \in \HH_{\bepsilon}(\wt{G}, I)$ is invertible for positive $y \in Y_{Q,n}$. Consider 
$$\rho_G =\frac{1}{2} \sum_{\alpha >0} \alpha,$$
the sum taken over all positive roots of $G$. 
For every $y\in Y_{Q,n}$, we write 
$$t_y^\theta:=q^{-\angb{y}{\rho_G}} e_{y_1}^\theta \cdot (e_{y_2}^\theta)^{-1},$$
where $y_1, y_2 \in Y_{Q,n}$ are any two positive elements such that $y= y_1 - y_2$. The elements in $\set{e_w^\theta: w\in W}$ and $\set{t_y^\theta: y\in Y_{Q,n}}$ then give the Bernstein presentation of $\HH_{\bepsilon}(\wt{G}, I)$ as follows, due to G. Savin \cite{Sav88, Sav04} (see also \cite{Mc1, GG}).

\begin{thm} \label{T:IH}
The Iwahori--Hecke algebra $\HH_{\bepsilon}(\wt{G}, I)$ has the following description:
$$\HH_{\bepsilon}(\wt{G}, I) = \big\langle t_y^\theta, e_{w_\alpha}^\theta:\  y\in Y_{Q,n}, \alpha^\vee \in \Delta^\vee \big\rangle$$
with relations given by
\begin{enumerate}
\item[(i)] $(e_{w_\alpha}^\theta-q)(e_{w_\alpha}^\theta+1)=0$.

\item[(ii)] $(e_{w_\alpha}^\theta e_{w_\beta}^\theta)^r=(e_{w_\beta}^\theta e_{w_\alpha}^\theta)^r$ if $w_\alpha w_\beta$ is of order $2r$. 

\item[(iii)] $(e_{w_\alpha}^\theta e_{w_\beta}^\theta )^r e_{w_\alpha}^\theta=(e_{w_\beta}^\theta e_{w_\alpha}^\theta )^r e_{w_\beta}^\theta$ if $w_\alpha w_\beta$ is of order $2r+1$.
\item[(iv)]  $t_y^\theta \cdot t_{y'}^\theta=t_{y+y'}^\theta$.
\item[(v)] Write $\langle y, \alpha\rangle=m n_\alpha$. Then
\begin{equation*}
    e_{w_\alpha}^\theta \cdot t_y^\theta =
    \begin{cases}
      t_{y^{w_\alpha}}^\theta \cdot e_{w_\alpha}^\theta + (q-1) \sum_{k=0}^{m-1} t_{y-kn_\alpha \alpha^\vee}^\theta & \text{if } m > 0,\\
      t_y^\theta \cdot e_{w_\alpha}^\theta & \text{if } m = 0,\\
      t_{y^{w_\alpha}}^\theta \cdot e_{w_\alpha}^\theta - (q-1) \sum_{k=0}^{1-m} t_{y+kn_\alpha \alpha^\vee}^\theta & \text{if } m < 0.
    \end{cases}
\end{equation*}
\end{enumerate}
In particular, $\HH_{\bepsilon}(\wt{G}, I)\simeq_\theta \C[Y_{Q,n}] \otimes_\C \mca{H}_W$ as $\C$-vector spaces, where the subalgebra $\mca{H}_W \subset \HH_{\bepsilon}(\wt{G}, I)$ is generated by $\set{e_{w_\alpha}^\theta: \alpha^\vee \in \Delta}$ and the commutative subalgebra $\C[Y_{Q,n}]$ generated by $\set{t_y^\theta: y\in Y_{Q,n}}$.
\end{thm}

Let $G_{Q,n}$ be the linear algebraic group whose root datum is given by
$$(X_{Q,n}, \Phi_{Q,n}; Y_{Q,n}, \Phi_{Q,n}^\vee).$$
In particular, we have 
$$G_{Q,n}^\vee \simeq \wt{G}^\vee,$$
where $G_{Q,n}^\vee$ is the Langlands dual of $G_{Q,n}$. Note that $\wt{W}_{\rm ex} = Y_{Q,n} \rtimes W$ is then the extended affine Weyl group for $G_{Q,n}$. Also, $G_{Q,n}$ has a natural choice of simple roots and simple coroots inherited from the choice of $\Delta$.
 Let $I_{Q,n} \subset G_{Q,n}$ be the Iwahori subgroup of $G_{Q,n}$, similar as $I \subset G$.
In the Iwahori--Hecke algebra $\mathcal{H}(G_{Q,n}, I_{Q,n})$, one has the natural element
$$E_w \in \mca{H}(G_{Q,n}, I_{Q,n})$$ for every 
$w\in \wt{W}_{\rm ex}$, as an analogue of $e_w^\theta \in \HH_{\bepsilon}(\wt{G}, I)$. Consider
$$\rho_{G_{Q,n}} = \frac{1}{2} \sum_{\alpha_{Q,n}>0} \alpha_{Q,n},$$
the half sum of positive roots of $G_{Q,n}$. For every $y\in Y_{Q,n}$, we define
$$T_y:= q^{-\angb{y}{\rho_{G_{Q,n}}}} E_{y_1} \cdot E_{y_2}^{-1},$$
where $y= y_1 - y_2$ with $y_1, y_2 \gest 0$. Then 
$$\mca{H}(G_{Q,n}, I_{Q,n}) \simeq \C[Y_{Q,n}] \otimes_\C \mca{H}_W$$
and is generated by $\set{E_{w_\alpha}: \alpha^\vee \in \Delta}$ and $\set{T_y: y\in Y_{Q,n}}$ with relations exactly in parallel with Theorem \ref{T:IH} (i)--(v). Thus, depending on the choice of the distinguished genuine character $\theta$ of $Z(\overline{T})$, we have an algebra isomorphism
\begin{equation} \label{D:iso-t}
\varphi_\theta: \HH_{\bepsilon}(\wt{G},  I) \longrightarrow \mca{H}(G_{Q,n}, I_{Q,n})
\end{equation}
given by 
$$\varphi_\theta(e_{w_\alpha}^\theta) = E_{w_\alpha} \text{ and } \varphi_\theta(t_y^\theta) = T_y.$$
This induces a natural map
\begin{equation} \label{E:phi*}
\varphi_\theta^*: \Irr(\HH_{\bepsilon}(\wt{G},  I)) \longrightarrow \Irr(\mca{H}(G_{Q,n}, I_{Q,n}))
\end{equation}
of irreducible left modules over the two algebras.

\subsection{LLC for $\Irr_\epsilon(\wt{G})^I$}
Let 
$$\Irr_\epsilon(\wt{G})^I \subset \Irr_\epsilon(\wt{G})$$
be the subset of irreducible genuine Iwahori-spherical representations $\pi$ of $\wt{G}$, i.e., $\pi^I \ne \set{0}$. 
By Borel and Casselman, one has a natural bijection
\begin{equation} \label{msc-E}
\msc{E}: \Irr_\epsilon(\wt{G})^I \longrightarrow \Irr_L(\mca{H}_{\bepsilon}(\wt{G}, I))
\end{equation}
given by 
$$\pi \mapsto \pi^I,$$
see \cite{Sav04, FlKa86} for a discussion for covers.

We will use $\varphi_\theta^*$ in \eqref{E:phi*} to transport a local Langlands correspondence for Iwahori-spherical representations of $G_{Q,n}$ to the cover $\wt{G}$. Thus, we recall a Deligne--Langlands correspondence
$$\mca{L}_{G_{Q,n}}: \Irr(\mca{H}(G_{Q,n}, I_{Q,n})) \longrightarrow {\rm Par}^{\rm DL}(\WD_F, G_{Q,n}^\vee)$$
established by Kazhdan--Lusztig \cite{KL2} for $G_{Q,n}$ with connected center, and extended to the general case by Reeder \cite{Ree4}. See also \cite{ABPS17} for a discussion of the local Langlands correspondence for constituents of (not necessarily unramified) principal series of a linear algebraic group. Here ${\rm Par}^{\rm DL}(\WD_F, G_{Q,n}^\vee)$ consists of pairs $(\phi,  \tau)$, where
\begin{equation} \label{E:phi}
\phi: \WD_F \longrightarrow G_{Q,n}^\vee
\end{equation}
is a ($G_{Q,n}^\vee$-conjugacy class of) group homomorphism that is trivial on the inertial subgroup $I_F \subset W_F$, and $\tau \in \Irr(\mca{S}_\phi)$ with 
$$\mca{S}_\phi: = \pi_0({\rm Cent}_{G_{Q,n}^\vee}(\phi)/Z(G_{Q,n}^\vee)).$$
Such an $\mca{L}_{G_{Q,n}}$ is injective and satisfies several natural properties. A pair $(\phi, \tau)$ lies in the image of $\mca{L}_{G_{Q,n}}$ if and only if $\tau$ occurs in certain homology space given as follows. 

Let $\mfr{B}$ denote the variety of Borel subgroups of $G_{Q,n}^\vee$, and let $\mfr{B}^\phi \subset \mfr{B}$ denote the subvariety consisting of Borel subgroups of $G_{Q,n}^\vee$ containing 
$$\phi(W_F \times B_2),$$
where $B_2 \subset \SL_2(\C)$ is the Borel subgroup of upper triangular matrices. The group ${\rm Cent}_{G_{Q,n}^\vee}(\phi)$ acts naturally on $\mfr{B}^\phi$ and thus on the homology $H_*(\mfr{B}^\phi, \C)$. The action factors through $\mca{S}_\phi$ and gives a representation of $\mca{S}_\phi$ on $H_*(\mfr{B}^\phi, \C)$. The requirement on $\tau \in \Irr(\mca{S}_\phi)$ such that $(\phi, \tau)$ lies in the image of $\mca{L}_{G_{Q,n}}$ is that it appears in $H_*(\mfr{B}^\phi, \C)$: such $\tau$ is called geometric.

Consider
$${\rm Par}^{\rm DL}_{\rm geo}(\WD_F, G_{Q,n}^\vee):=\set{(\phi, \tau): \ \tau \text{ is geometric}} \subset {\rm Par}^{\rm DL}(\WD_F, G_{Q,n}^\vee).$$
Then one has a bijection
\begin{equation} \label{E:L-DL}
\mca{L}_{G_{Q,n}}^{\rm DL}: \Irr(\mca{H}(G_{Q,n}, I_{Q,n})) \longrightarrow {\rm Par}^{\rm DL}_{\rm geo}(\WD_F, G_{Q,n}^\vee)
\end{equation}
arising from $\mca{L}_{G_{Q,n}}$.

We also recall the set  ${\rm Par}^{\rm KLR}(G_{Q,n}^\vee)$ of Kazhdan--Lusztig--Reeder (KLR) parameters, which are often convenient to compute. Consider the pair $(s, u)$, where $s \in T_{Q,n}^\vee$ and $u \in G_{Q,n}^\vee$ is unipotent such that 
$$s u s^{-1} = u^q.$$
We have the component group
$$\mca{S}_{s, u}:=\pi_0({\rm Cent}_{G_{Q,n}^\vee}(s, u)/Z(G_{Q,n}^\vee))$$
of $s, u$. A KLR parameter is just $(s, u, \tau)$ with $\tau \in \Irr(\mca{S}_{s, u})$. Let $\mfr{B}^{s, u} \subset \mfr{B}$ be the subvariety consisting of all Borel subgroups of $G_{Q,n}^\vee$ that contain both $s$ and $u$. Again,  ${\rm Cent}_{G_{Q,n}^\vee}(s, u)$ acts naturally on $H_*(\mfr{B}^{s, u}, \C)$ and the action factors through $\mca{S}_{s, u}$. A representation $\tau \in \Irr(\mca{S}_{s, u})$ is called geometric if it appears in $H_*(\mfr{B}^{s, u}, \C)$.
This gives the subset
$${\rm Par}^{\rm KLR}_{\rm geo}(G_{Q,n}^\vee) \subset {\rm Par}^{\rm KLR}(G_{Q,n}^\vee)$$
consisting of $(s, u, \tau)$ where $\tau$ is geometric.

The map
$$\phi \mapsto (s_\phi,  u_\phi),$$
where 
$$s_\phi:= \phi\left(\varpi, \left(\begin{array}{cc} q^{1/2} & 0 \\ 0 & q^{-1/2} \end{array}\right) \right) \text{ and } u_\phi:=\phi \left(1, \left(\begin{array}{cc} 1 & 1 \\ 0 & 1 \end{array}\right) \right),$$
gives a bijection between the $\phi$ as in \eqref{E:phi} and such pair $(s, u)$.  If $\phi$ corresponds to $(s, u)$ as above, then 
$$\mca{S}_{s, u} \simeq \mca{S}_\phi.$$
Thus, one has a bijection
$${\rm Par}^{\rm DL}_{\rm geo}(\WD_F, G_{Q,n}^\vee) \longrightarrow  {\rm Par}^{\rm KLR}_{\rm geo}(G_{Q,n}^\vee)$$
given by $(\phi, \tau) \mapsto (s_\phi, u_\phi, \tau)$. Accordingly, one has a bijection
$$\mca{L}_{G_{Q,n}}^{\rm KLR}: \Irr(\mca{H}(G_{Q,n}, I_{Q,n})) \longrightarrow {\rm Par}^{\rm KLR}_{\rm geo}(G_{Q,n}^\vee).$$
In fact, $\mca{L}_{G_{Q,n}}^{\rm KLR}$ was the form established by Kazhdan--Lusztig \cite{KL2} and Reeder \cite{Ree4}.

Since $G_{Q,n}^\vee \simeq \wt{G}^\vee$, we can view 
$$(\phi, \tau) \in {\rm Par}^{\rm DL}(\WD_F, \wt{G}^\vee)$$
and want to define an analogue of such parameters valued in ${}^L \wt{G}$ instead of $\wt{G}^\vee$. 
For every $\xi: \WD_F \longrightarrow {}^L\wt{G}$, we define
$$\mca{S}_{\xi} := \pi_0({\rm Cent}_{{}^L \wt{G} }(\xi)/Z(\wt{G}^\vee)).$$
Also, denote by ${\rm Par}^{\rm DL}(\WD_F, {}^L \wt{G})$ the set of pairs $(\xi, \gamma)$ where $\xi$ is a group homomorphism as above and $\gamma \in \Irr(\mca{S}_\xi)$.


Recall the isomorphism
$$i_\theta: {}^L \wt{G} \simeq  \wt{G}^\vee \times \WD_F $$
from \eqref{D:LGiso}, which then gives an embedding of $\wt{G}^\vee \into {}^L \wt{G}$ given by $g \mapsto i_\theta^{-1}(g, 1)$, and thus a parameter
$$i_\theta(\phi): \WD_F \longrightarrow {}^L\wt{G}$$
given by
$$i_\theta(\phi)(a)=i_\theta^{-1}(\phi(a), a).$$
It is easy to see that one has 
$$\mca{S}_{i_\theta(\phi)} \simeq \mca{S}_{\phi}.$$
By transport of structure, this gives a representation $i_\theta(\tau) \in \Irr(\mca{S}_{i_\theta(\phi)})$ for every $\tau \in \Irr(\mca{S}_\phi)$. We thus have a map
\begin{equation} \label{E:i*}
i_\theta^*: {\rm Par}^{\rm DL}(\WD_F, \wt{G}^\vee) \longrightarrow {\rm Par}^{\rm DL}(\WD_F, {}^L\wt{G})
\end{equation}
given by
$$i_\theta^*(\phi, \tau):=(i_\theta(\phi), i_\theta(\tau)).$$

Consider the following composite
\begin{equation}
\begin{tikzcd}
\mca{L}_{\wt{G}}^{\rm DL}=i_\theta^* \circ \mca{L}_{G_{Q,n}}^{\rm DL} \circ \varphi_\theta^* \circ \msc{E}: \Irr_\epsilon(\wt{G})^I \ar[r] & {\rm Par}^{\rm DL}(\WD_F, {}^L\wt{G}),
\end{tikzcd}
\end{equation}
where $\msc{E}, \varphi_\theta^*, \mca{L}_{G_{Q,n}}^{\rm DL}$ and $i_\theta^*$ are given in \eqref{msc-E}, \eqref{E:phi*}, \eqref{E:L-DL} and \eqref{E:i*} respectively. Note that $\mca{L}^{\rm DL}_{G_{Q,n}}$ is not completely canonical, and the simplest such example arises from the unitary unramified principal series $I(\chi)$ of $\SL_2$ that is reducible. For more discussions on this issue, see \cite[\S 14]{ABPS17}. Thus, $\mca{L}_{\wt{G}}^{\rm DL}$ is not completely canonical. However, the dependence on $\theta$ in the definition of $\mca{L}_{\wt{G}}^{\rm DL}$ can be eliminated, as follows.

\begin{prop} \label{P:indep}
For a fixed Deligne--Langlands correspondence $\mca{L}_{G_{Q,n}}^{\rm DL}$ of $G_{Q,n}$, the map $\mca{L}_{\wt{G}}^{\rm DL}$ is independent of the choice of any unramified distinguished genuine character $\theta$.
\end{prop}
\begin{proof}
 We first show that the ${}^L\wt{G}$-valued L-parameter associated with $\pi$ via $\mca{L}_{\wt{G}}^{\rm DL}$ is independent of $\theta$.

Suppose we have unramified distinguished genuine characters $\theta_j$ of $Z(\wt{T})$ with $1\lest j \lest 2$, and thus also KLR parameter $(s_{\theta_i}, u_{\theta_i}, \tau_{\theta_i})$ relative to $\theta_i$. Consider the isomorphism 
$$\varphi= \varphi_{\theta_2} \circ \varphi_{\theta_1}^{-1}: \HH(G_{Q,n}, I_{Q,n}) \longrightarrow \HH(G_{Q,n}, I_{Q,n}).$$ 
To understand the relation between $(s_{\theta_1}, u_{\theta_1}, \tau_{\theta_1})$ and $(s_{\theta_2}, u_{\theta_2}, \tau_{\theta_2})$, it suffices to look at the effect of $\varphi$ on $\mca{L}_{G_{Q,n}}^{\rm DL}$.   By the definition of $e_w^{\theta_i} \in \HH_{\bepsilon}(\wt{G}, I)$ we see that 
$$\varphi(E_w) =E_w$$
if $w\in W$. On the other hand, for $y\in Y_{Q,n}$ dominant, we have
$$e_y^{\theta_1} =\frac{\theta_2}{\theta_1}(\varpi^y) \cdot e_y^{\theta_2} \in \HH_{\bepsilon}(\wt{G}, I).$$
If $y=x-z$ with positive $x, z\in Y_{Q,n}$, then
$$q^{\angb{y}{\rho_G}} \cdot t_y^{\theta_1} = e_x^{\theta_1} \cdot (e_z^{\theta_2})^{-1} = \frac{\theta_2}{\theta_1}(\varpi^x) \cdot e_x^{\theta_2} \cdot \left( \frac{\theta_2}{\theta_1}(\varpi^z) \cdot e_z^{\theta_2} \right)^{-1} = \frac{\theta_2}{\theta_1}(\varpi^y) \cdot t_y^{\theta_2} \cdot q^{\angb{y}{\rho_G}}.$$
Thus, $t_y^{\theta_1}=(\theta_2/\theta_1)(\varpi^y) \cdot t_y^{\theta_2}$. Then we get
$$\varphi(T_y) = (\theta_2/\theta_1)(\varpi^y) \cdot T_y$$
for every $y\in Y_{Q,n}$. Associated with $\theta_2/\theta_1$ is an element
$$z_{\theta_2/\theta_1} \in Z(G_{Q,n}^\vee)\simeq Z(\wt{G}^\vee)$$
given by $z_{\theta_2/\theta_1}(y):= (\theta_2/\theta_1)(\varpi^y)$ for every $y \in Y_{Q,n}$.
Consider the map 
$$\varphi^*: \Irr(\HH(G_{Q,n}, I_{Q,n})) \longrightarrow \Irr(\HH(G_{Q,n}, I_{Q,n}))$$
 induced from $\varphi$. Let $\sigma(s, u, \tau):= (\mca{L}_{G_{Q,n}}^{\rm KLR})^{-1}(s, u, \tau) \in \Irr(\HH(G_{Q,n}, I_{Q,n}))$ be associated with any KLR parameter $(s, u, \tau) \in {\rm Par}^{\rm KLR}_{\rm geo}(G_{Q,n}^\vee)$.
 Then we have
\begin{equation} \label{ch-par}
\varphi^*(\sigma(s, u, \tau)) = \sigma(s\cdot z_{\theta_2/\theta_1}, u, \tau),
\end{equation}
see \cite[Page 843, (114)]{ABPS17}. This shows that
\begin{equation} \label{rel-par}
s_{\theta_2} = s_{\theta_1} \cdot z_{\theta_2/\theta_1},\quad u_{\theta_2} = u_{\theta_1}, \quad \tau_{\theta_2} = \tau_{\theta_1}.
\end{equation}

Denote by 
$$\phi_{\theta_2/\theta_1}: \WD_F \to Z(\wt{G}^\vee) \into \wt{G}^\vee$$
the unique homomorphism that is trivial on $\SL_2(\C)$ and $I_F \subset W_F$, and such that  $\phi_{\theta_2/\theta_1}({\rm Frob}) =z_{\theta_2/\theta_1}$. It is easy to see that the diagram
$$\begin{tikzcd}
\wt{G}^{\vee} \times \WD_F \ar[r, "{i_{\theta_1}}"] \ar[d, "f"] & {}^L\wt{G} \ar[d, equal] \\
\wt{G}^{\vee} \times \WD_F \ar[r, "{i_{\theta_2}}"] & {}^L\wt{G}
\end{tikzcd}
$$
 commutes, where $f(g, x):=(g\cdot \phi_{\theta_2/\theta_1}(x), x)$. Combining the relations in \eqref{rel-par} and the above diagram, it follows that
$$i_{\theta_1}(\phi_{s_{\theta_1}, u_{\theta_1}}) = i_{\theta_2}(\phi_{s_{\theta_2}, u_{\theta_2}}).$$
This shows that the ${}^L\wt{G}$-valued L-parameter associated with $\pi$ via $\mca{L}_{\wt{G}}^{\rm DL}$ is independent of $\theta$, and we write it as $\phi_\pi$.

From the above discussion we have canonical isomorphisms
\begin{equation} \label{E:S1}
\mca{S}_{s_{\theta_1}, u} \simeq \mca{S}_{\phi_\pi} \simeq \mca{S}_{s_{\theta_2}, u}.
\end{equation}
Also, $\tau_{\theta_1}$ and $\tau_{\theta_2}$ correspond to each other via the map induced from \eqref{E:S1}. Hence, they gives rise to the same element in $\Irr(\mca{S}_{\phi_\pi})$.

In view of the equivalence between $\mca{L}_{G_{Q,n}}^{\rm DL}$ and $\mca{L}_{G_{Q,n}}^{\rm KLR}$, the above completes the proof.
\end{proof}

\begin{dfn} \label{DLcor}
For fixed $\mca{L}_{G_{Q,n}}^{\rm DL}$, we call
$$\mca{L}_{\wt{G}}^{\rm DL}: \Irr_\epsilon(\wt{G})^I \longrightarrow {\rm Par}^{\rm DL}(\WD_F, {}^L\wt{G}), \quad \pi \mapsto (\phi_\pi, \tau_\pi)$$
the Deligne--Langlands correspondence for the irreducible genuine Iwahori-spherical representations of $\wt{G}$ associated with $\mca{L}_{G_{Q,n}}^{\rm DL}$.
\end{dfn}

As in the linear algebraic case, one can also  define the Kazhdan--Lusztig--Reeder parameter $(s_\pi, u_\pi, \tau_\pi)$ associated with $(\phi_\pi, \tau_\pi)$, now with $s\in {}^L\wt{G}$ instead of in $\wt{G}^\vee$. 
By the local Langlands correspondence for covering torus (see \cite[\S 15]{We6} or \cite[\S 8]{GG}), one has the ``absolute" Satake parameter 
$$s_\chi:=\phi_\chi({\rm Frob}) \in {}^L \wt{T} \subset {}^L \wt{G},$$
where $\phi_\chi: W_F \to {}^L\wt{T}$ is the $L$-parameter associated with $\chi$. 
Let $(s_\theta, u_\theta, \tau_\theta)$ be the KLR parameter relative to $\theta$.  Then $s_\theta \in \wt{T}^\vee$ is the ``relative" Satake parameter of $s_\chi$, i.e., the one obtained from projecting $i_{\theta}(s_\chi)$ to the $\wt{G}^\vee$-component.  
The sought $(s_\pi, u_\pi, \tau_\pi)$ is just $(s_\chi, u_\theta, \tau_\theta)$, where the independence of $u_\theta$ and $\tau_\theta$ on $\theta$ is clear from the proof of Proposition \ref{P:indep}.

\subsection{L-functions and gamma-factors} \label{SS:gamma}
Let 
$$R: {}^L \wt{G} \longrightarrow \GL(V)$$
be a finite dimensional algebraic representation afforded on a finite-dimensional complex vector space $V$, i.e., its restriction to $\wt{G}^\vee$ is an algebraic group homomorphism. For every $\pi \in \Irr(\wt{G})^I$ one can define the L-function
$$L(s, \pi, R):=L(s, R \circ \phi_\pi),$$
where the right hand side denotes the local Artin L-function, see \cite{Tat1, Roh}.
Moreover, Langlands and Deligne defined the $\varepsilon$-factor $\varepsilon(s, R \circ \phi_\pi, \psi)$, which is uniquely characterized by certain properties (see loc. cit.), and thus one has the gamma-factor
$$\gamma(s, \pi, R, \psi):= \gamma(s, R \circ \phi_\pi, \psi) =\varepsilon(s, R \circ \phi_\pi, \psi) \cdot \frac{L(1-s, R^\vee \circ \phi_\pi)}{L(s, R \circ \phi_\pi)},$$
where $R^\vee$ denotes the contragredient representation of $R$.

By the construction of ${}^L \wt{G}$, one has a canonical isomorphism
$${}^L \wt{G}/Z(\wt{G}^\vee) \simeq (\wt{G}^\vee/Z(\wt{G}^\vee)) \times \WD_F.$$
Thus if $R={\rm Ad}$ is the adjoint representation of ${}^L\wt{G}$ on ${\rm Lie}(\wt{G}^\vee)$, then it factors through
$${\rm Ad}^\dag: \wt{G}^\vee/Z(\wt{G}^\vee) \longrightarrow \GL({\rm Lie}(\wt{G}^\vee)).$$
In particular, this means that for any distinguished genuine character $\theta$ one has
\begin{equation} \label{gam-red}
L(s, \pi, {\rm Ad}) = L(s, {\rm Ad}^\dag \circ \phi_{\pi, \theta}), \quad \gamma(s, \pi, {\rm Ad}, \psi) = \gamma(s, {\rm Ad}^\dag \circ \phi_{\pi, \theta})
\end{equation}
where $\phi_{\pi, \theta}:= \mca{L}_{G_{Q,n}}^{\rm DL} \circ \varphi_\theta^* \circ \msc{E}(\pi) \in {\rm Par}^{\rm DL}(\WD_F, \wt{G}^\vee)$ is the L-parameter of $\pi$ relative to $\theta$. Here $\phi_{\pi, \theta}$ is valued in $\wt{G}^\vee$, but we also naturally view it as valued in $\wt{G}^\vee/Z(\wt{G}^\vee)$ by abuse of notation.

\section{Formal degrees and the Hiraga--Ichino--Ikeda formula} \label{S:fd}
We retain the tame setting as in \S \ref{SS:HH} and let $(\sigma, V_\sigma)$ be an irreducible (and thus necessarily finite-dimensional) representation of $\HH_{\bepsilon}(\wt{G}, I)$.
Consider the irreducible genuine Iwahori-spherical representation 
$$\pi_\sigma:=\msc{E}^{-1}(\sigma) \in \Irr_\epsilon(\wt{G})$$ afforded by a certain vector space $V_{\pi_\sigma}$, where $\msc{E}$ is as in \eqref{msc-E}. Let $(\sigma^\vee, V_\sigma^\vee)$ be the contragredient representation of  $(\sigma, V_\sigma)$ specified
$$\angb{v}{\sigma^\vee(u)(\lambda)} = \angb{\sigma(\check{u})v}{\lambda}$$
where $\check{u}(g):=u(g^{-1})$ for all $u \in \HH_{\bepsilon}(\wt{G}, I)$. Note that $\check{u} \in \HH_{\epsilon}(\wt{G}, I)$ and thus $\sigma^\vee$ is an irreducible representation of $\HH_{\epsilon}(\wt{G}, I)$. Let $(\pi_\sigma^\vee, V_{\pi_\sigma}^\vee)$ be the contragredient representation of $(\pi_\sigma, V_{\pi_\sigma})$ given in 
\S \ref{SS:ds-fd}.

The isomorphism $\varphi_\theta$ gives rise to an irreducible representation $\varphi_\theta^*(\sigma)$ of $\HH(G_{Q,n}, I_{Q,n})$. For simplicity of notation, we write
$$\sigma_\theta:=\varphi_\theta^*(\sigma).$$
This gives $(\sigma_\theta, V_\sigma) \in \Irr(\HH(G_{Q,n}, I_{Q,n}))$, where the representation $\sigma_\theta$ is realized on the space $V_\sigma$. Let $(\pi_{\sigma_\theta}, V_{\pi_{\sigma_\theta}}) \in \Irr(G_{Q,n})$ be the irreducible $I_{Q,n}$-spherical representation associated with $\sigma_\theta$. It is clear that there is a vector space isomorphism
$$\eta: (\pi_\sigma)^I \longrightarrow (\pi_{\sigma_\theta})^{I_{Q,n}}$$
that is equivariant with respect to the $\HH_{\bepsilon}(\wt{G}, I)$ and $\HH(G_{Q,n}, I_{Q,n})$ actions on the domain and codomain respectively.

We choose $v\in (\pi_\sigma)^I$ and set $v':=\eta(v) \in (\pi_{\sigma_\theta})^{I_{Q,n}}$. The canonical pairing between $\pi_{\sigma_\theta}$ and $\pi_{\sigma_\theta}^\vee$ gives an isomorphism (see \cite[Proposition 4.2.5]{Bum})
$$((\pi_{\sigma_\theta})^{I_{Q,n}})^\vee  \simeq (\pi_{\sigma_\theta}^\vee)^{I_{Q,n}}.$$
Let $\lambda' \in (\pi_{\sigma_\theta}^\vee)^{I_{Q,n}}$ be such that  $\angb{v'}{\lambda'} \ne 0$. It gives an element $\angb{\eta(\cdot)}{\lambda'} \in ((\pi_\sigma)^I)^\vee$. Again, the canonical isomorphism
$$((\pi_\sigma)^I)^\vee \simeq (\pi_\sigma^\vee)^I$$
gives a $\lambda \in (\pi_\sigma^\vee)^I$ such that $\angb{\cdot}{\lambda} = \angb{\eta(\cdot)}{\lambda'}$. That is, we have the following commutative diagram
$$\begin{tikzcd}
(\pi_\sigma)^I \ar[r, "\eta"]  \ar[rd, "\lambda"] & (\pi_{\sigma_\theta})^{I_{Q,n}} \ar[d, "{\lambda'}"] \\
& \C,
\end{tikzcd}$$
which satisfies 
\begin{equation}  \label{E:eq-p}
\angb{\sigma(u)(\cdot)}{\lambda} = \angb{\sigma_\theta(\varphi_\theta(u)) \eta(\cdot)}{\lambda'}
\end{equation}
for all $u\in \HH_{\bepsilon}(\wt{G}, I)$. In particular, $\angb{v}{\lambda} = \angb{v'}{\lambda'}$.

\begin{lm} \label{L:2fdeg}
For $v \in \pi_\sigma^I, \lambda \in (\pi_\sigma^\vee)^I$ as above, one has
$$\deg(\pi_\sigma, \mu_{\wt{G}})^{-1} =\frac{1}{\angb{v}{\lambda}^2} \sum_{w\in \wt{W}_{\rm ex}}  q^{-l_G(w)}  \cdot \angb{\sigma(e_w^\theta)(v)}{\lambda} \cdot \angb{\sigma(e_{w^{-1}}^\theta)(v)}{\lambda},$$
where $\mu_{\wt{G}}$ is the measure of $\wt{G}$ such that $\mu_{\wt{G}}(\pmb{\delta}_n \cdot I)=1$.
\end{lm}
\begin{proof}
Since $v$ and $\lambda$ are both $I$-fixed, the matrix coefficients $f_{v, \lambda}$ is $I$-biinvariant. Thus, in view of the definition of $\deg(\pi_\sigma, \mu_{\wt{G}})$ in \eqref{E:fd}, it suffices to show that
\begin{equation} \label{E:key}
\int_{\pmb{\delta}_n I \wt{w} I} \angb{ \pi_\sigma(g^{-1})(v) }{\lambda} \cdot  \angb{ \pi_\sigma(g)(v) }{\lambda} 
 = q^{-l_G(w)}  \cdot \angb{\sigma(e_w^\theta)(v)}{\lambda} \cdot \angb{\sigma(e_{w^{-1}}^\theta)(v)}{\lambda} .
\end{equation}
The left hand side of \eqref{E:key} is equal to
$$q^{l_G(w)} \cdot \angb{ \pi_\sigma(\wt{w}^{-1})(v) }{\lambda} \cdot  \angb{ \pi_\sigma(\wt{w})(v) }{\lambda}.$$
On the other hand,
$$\begin{aligned}
\angb{\sigma(e_w^\theta)}{\lambda} & = \int_{\wt{G}}  e_w^\theta(g) \cdot \angb{\pi_\sigma(g)(v)}{\lambda}\ \mu_{\wt{G}}(g) \\
& = q^{l_G(w)} \cdot \theta^{-1}(\wt{w}) \cdot \angb{\pi_\sigma(\wt{w})}{\lambda}.
\end{aligned}$$
Similarly, we have
$$\angb{\sigma(e_{w^{-1}}^\theta)}{\lambda} = q^{l_G(w)} \cdot \theta(\wt{w}) \cdot \angb{\pi_\sigma(\wt{w}^{-1})}{\lambda}.$$
Combining the above, one obtains \eqref{E:key} and completes the proof.
\end{proof}

Recall that $\wt{W}_{\rm ex}=Y_{Q,n} \times W$ is the extended affine Weyl group of $G_{Q,n}$. Thus, one has the length function
$$l_{G_{Q,n}}: \wt{W}_{\rm ex} \longrightarrow \N$$
which is explicitly given by
\begin{equation} \label{l_GQn}
l_{G_{Q,n}}(w)=\sum_{\alpha \in \Phi^+(s)} \val{\angb{n_\alpha^{-1}\alpha}{y}} + \sum_{\alpha \in \Phi^-(s)} \val{\angb{n_\alpha^{-1}\alpha}{y}+ 1}.
\end{equation}
Here $w=(y, s) \in \wt{W}_{\rm ex}$ with $y\in Y_{Q,n}, s\in W$ and $\Phi^\pm(s)$ are the same as in \eqref{l_G}. It follows from Lemma \ref{L:2fdeg} that for $\pi_{\sigma_\theta} \in \Irr(G_{Q,n})$ one has analogously the equality
$$\deg(\pi_{\sigma_\theta}, \mu_{G_{Q,n}})^{-1} =\frac{1}{\angb{v'}{\lambda'}^2} \sum_{w\in \wt{W}_{\rm ex}}  q^{-l_{G_{Q,n}}(w)}  \cdot \angb{\sigma_\theta(E_w)(v')}{\lambda'} \cdot \angb{\sigma_\theta(E_{w^{-1}})(v')}{\lambda'},$$
where $\mu_{G_{Q,n}}$ is the measure such that $\mu_{G_{Q,n}}(I_{Q,n})=1$. In fact, the right hand side of the above equality is often taken as the definition of formal degree of discrete series representation of the Iwahori--Hecke algebra of a linear group, see \cite[Page 37]{Mat77-B} or the discussion in \cite[\S 4.1]{FeOp20}.

The following proposition plays a crucial role in the comparison between $\deg(\pi_\sigma, \mu_{\wt{G}})$ and $\deg(\pi_{\sigma_\theta}, \mu_{G_{Q,n}})$.

\begin{prop} \label{P:key-HII}
For every $w\in \wt{W}_{\rm ex}$, one has
\begin{equation} \label{phi-E}
q^{-l_G(w)/2} \cdot \varphi_\theta(e_w^\theta) = q^{-l_{G_{Q,n}}(w)/2} \cdot E_w
\end{equation}
and thus
\begin{equation} \label{e=E}
q^{-l_G(w)}  \cdot \angb{\sigma(e_w^\theta)(v)}{\lambda} \cdot \angb{\sigma(e_{w^{-1}}^\theta)(v)}{\lambda} = q^{-l_{G_{Q,n}}(w)}  \cdot \angb{\sigma_\theta(E_w)(v')}{\lambda'} \cdot \angb{\sigma_\theta(E_{w^{-1}})(v')}{\lambda'}.
\end{equation}
\end{prop}
\begin{proof}
 We check \eqref{phi-E} in three steps by considering: 
\begin{enumerate}
\item[(i)] when $w \in W$, 
\item[(ii)] when $w=y \in Y_{Q,n}^+$ is dominant, 
\item[(iii)] when $w=(y, s) \in \wt{W}_{\rm ex}$ is a general element.
\end{enumerate}

First, the equality \eqref{phi-E} in case (i) follows from \eqref{E:eq-p} and the fact that $l_G(w) = l_{G_{Q,n}}(w)$ and $\varphi_\theta(e_w^\theta) = E_w$ for $w\in W$. For case (ii) we note that for $y\in Y_{Q,n}^+$ dominant, one has  $\varphi_\theta(t_y^\theta) = T_y$
and thus
$$q^{-\angb{y}{\rho_G}} \cdot \varphi_\theta(e_y^\theta) = q^{-\angb{y}{\rho_{G_{Q,n}}}} \cdot E_y.$$
Since $l_G(y) = 2\angb{y}{\rho_G}$ and similarly $l_{G_{Q,n}}(y) = 2\angb{y}{\rho_{G_{Q,n}}}$, the equality \eqref{phi-E} is clear in this case.

Now we consider case (iii) and thus a general $w = (y, s) \in \wt{W}_{\rm ex} = Y_{Q,n} \rtimes W$ with $y\in Y_{Q,n}$ and $s\in W$. We write for brevity $w=ys$ and recall an argument from the proof of \cite[Proposition 6.3]{Sav04}. If $y$ is not positive, then there exists a simple root $\alpha$ such that $\angb{\alpha}{y} < 0$. Using the formula of $l_G$ in \eqref{l_G}, we have
$$l_G(w_\alpha \cdot w) = l_G(w) + 1.$$
Inductively, one can find $s' \in W$ such that $s' w = z s''$ with $z \in Y_{Q,n}^+$ and $s''\in W$ satisfying
\begin{equation} \label{l-eq}
l_G(s' w) = l_G(z) + l_G(s'') = l_G(s') + l_G(w).
\end{equation}
Note that $\angb{\alpha}{y}<0$ if and only if $\angb{n_\alpha^{-1} \alpha}{y}<0$, and thus one also has $l_{G_{Q,n}}(w_\alpha \cdot w) = l_{G_{Q,n}}(w) + 1$ by checking \eqref{l_GQn} instead. Hence, for the above $s'\in W$ it also gives
\begin{equation} \label{lc-eq}
l_{G_{Q,n}}(s' w) = l_{G_{Q,n}}(z) + l_{G_{Q,n}}(s'') = l_{G_{Q,n}}(s') + l_{G_{Q,n}}(w).
\end{equation}

It follows from \eqref{fac-e} and \eqref{l-eq} that
$$q^{-l_G(s')/2} q^{-l_G(w)/2} \cdot \varphi_\theta(e_{s'}^\theta) \cdot \varphi_\theta(e_w^\theta) = q^{-l_G(z)/2} q^{-l_G(s'')/2} \cdot \varphi_\theta(e_{z}^\theta) \cdot \varphi_\theta(e_{s''}^\theta).$$
Similarly, by using \eqref{fac-e} and \eqref{lc-eq} one has
$$q^{-l_{G_{Q,n}}(s')/2} q^{-l_{G_{Q,n}}(w)/2} \cdot E_{s'} \cdot E_w = q^{-l_{G_{Q,n}}(z)/2} q^{-l_{G_{Q,n}}(s'')/2} \cdot E_{z} \cdot E_{s''}.$$
Applying (i) above to $s', s''$ and (ii) to $z$, and noting the fact that $E_{s'}$ is invertible, we obtain \eqref{phi-E} for $w$.

Lastly, equality \eqref{phi-E} coupled with \eqref{E:eq-p} give
\begin{equation} 
q^{-l_G(w)/2} \cdot \angb{\sigma(e_w^\theta)(v)}{\lambda} = q^{-l_{G_{Q,n}}(w)/2}  \cdot \angb{\sigma_\theta(E_w)(v')}{\lambda'} 
\end{equation}
for every $w\in \wt{W}_{\rm ex}$. A similar equality holds for $w^{-1}$ in place of $w$. We have  $l_G(w) = l_G(w^{-1})$ for every $w\in \wt{W}_{\rm ex}$, which then gives \eqref{e=E}. This completes the proof.
\end{proof}

Let $\psi: F \to \C^\times$ be a non-trivial character. 
Let $\mu_{G, \psi}^\natural$ be the canonical measure on $G$ as given in \cite[\S 4]{Gro97}, using a certain differential form of top degree on $G$ over $F$ and self-dual Haar measure on $F$ with respect to $\psi$. Let $\mu_{\wt{G}, \psi}^\natural$ be the unique Haar measure on $\wt{G}$ such that
\begin{equation} \label{E:meas}
\mfr{q}_*(\mu_{\wt{G}, \psi}^\natural)=\mu_{G, \psi}^\natural,
\end{equation}
where the left hand side denotes the push-out measure of $\mu_{\wt{G}, \psi}^\natural$ via the quotient map $\mfr{q}: \wt{G} \onto G$. 
In particular, we have 
$$\mu_{\wt{G}, \psi}^\natural(\pmb{\delta}_n \times I) = \mu_{G, \psi}^\natural(I).$$ 
If $\psi$ is of conductor $O_F$, then
$$\mu_{\wt{G}, \psi}^\natural(\pmb{\delta}_n \times I) = \mu_{G, \psi}^\natural(I) = q^{-(l + \dim \mbf{G})/2} \cdot \val{\mbf{T}(\kappa_F)} = q^{-\val{\Phi^+}} \cdot (1-q^{-1})^l,$$
where the middle equality follows from \cite[(4.11)]{Gro97} or \cite[Page 633]{Kot88}.

\begin{thm} \label{T:fd-HII}
Let $\pi_\sigma = \msc{E}^{-1}(\sigma) \in \Irr_\epsilon(\wt{G})$ be associated with $\sigma \in \Irr(\HH_{\bepsilon}(\wt{G}, I))$. Consider the Deligne--Langlands correspondence $\mca{L}_{G_{Q,n}}^{\rm DL}$ for $G_{Q,n}$ as given in \cite[Theorem 3.1]{FOS22}. Let $\mca{L}_{\wt{G}}^{\rm DL}$ be the Deligne--Langlands correspondence associated with $\mca{L}_{G_{Q,n}}^{\rm DL}$ as in Definition \ref{DLcor}.
Let $(\phi_\sigma, \tau_\sigma) = \mca{L}_{\wt{G}}^{\rm DL}(\pi_\sigma) \in {\rm Par}^{\rm DL}(\WD_F, {}^L\wt{G})$ be the Deligne--Langlands parameter of $\pi_\sigma$. Assume that $\psi$ is of conductor $O_F$ and $\pi_\sigma$ is a discrete series representation.
Then 
\begin{equation} \label{E:main}
\deg(\pi_\sigma, \mu_{\wt{G}, \psi}^\natural) = \frac{\dim \tau_\sigma}{ \val{ \mca{S}_{\phi_\sigma} } } \cdot \val{\gamma(0, \phi_\sigma, {\rm Ad}, \psi)},
\end{equation}
where $\gamma(0, \phi_\sigma, Ad, \psi)$ is given as in \S \ref{SS:gamma}.
\end{thm}
\begin{proof}
Recall the measure $\mu_{\wt{G}}$ satisfies $\mu_{\wt{G}}(\pmb{\delta}_n \times I) = 1$, and Proposition \ref{P:key-HII} gives
$$
\deg(\pi_\sigma, \mu_{\wt{G}}) = \deg(\pi_{\sigma_\theta}, \mu_{G_{Q,n}})$$
where $\mu_{G_{Q,n}}$ is such that $\mu_{G_{Q,n}}(I_{Q,n})=1$.

Let $\mu_{G_{Q,n}, \psi}^\natural$ be the canonical measure defined similarly as $\mu_{G, \psi}^\natural$. Let $\wt{c}, c\in \C^\times$ be constants defined by
$$\mu_{\wt{G}, \psi}^\natural = \wt{c} \cdot \mu_{\wt{G}}, \quad \mu_{G_{Q,n}, \psi}^\natural = c \cdot \mu_{G_{Q,n}}.$$
We get $\wt{c} = c= q^{-\val{\Phi^+}} \cdot (1-q^{-1})^l$.
This gives that
$$ \deg(\pi_\sigma, \mu_{\wt{G}, \psi}^\natural) = \wt{c}^{-1} \cdot \deg(\pi_\sigma, \mu_{\wt{G}})   = c^{-1} \cdot \deg(\pi_{\sigma_\theta}, \mu_{G_{Q,n}})  =  \deg(\pi_{\sigma_\theta}, \mu_{G_{Q,n}, \psi}^\natural).
$$

Let $\mca{L}_{G_{Q,n}}^{\rm DL}$ be as in \cite[Theorem 3.1]{FOS22}. Let $\pi'$ be a general Iwahori-spherical discrete series of $G_{Q,n}$ with Deligne--Langlands parameter $(\phi_{\pi'}, \tau_{\pi'}):=\mca{L}_{G_{Q,n}}^{\rm DL}(\pi')$. The Hiraga--Ichino--Ikeda formula (as given in \cite{HII, HIIc}) for such $\pi'$ and $\psi$ is proved in \cite[Theorem 2]{FOS22} and one has
\begin{equation} \label{E:fd-lin}
\deg(\pi', \mu_{G_{Q,n}, \psi}^\natural) = \frac{\dim \tau_{\pi'}}{\val{\mca{S}_{\phi_{\pi'}}}} \cdot \val{ \gamma(0, \phi_{\pi'}, {\rm Ad}, \psi) },
\end{equation}
where ${\rm Ad}: \wt{G}^\vee = G_{Q,n}^\vee \to \GL({\rm Lie}(\wt{G}^\vee))$ is the adjoint representation of $\wt{G}^\vee$. Now, if $\pi' = \pi_{\sigma_\theta}$, then we have $\mca{S}_{\phi_\sigma} = \mca{S}_{\phi_{\pi'}}$ and $\tau_{\sigma} = \tau_{\pi'}$ by the definition of $\mca{L}_{\wt{G}}^{\rm DL}$. Moreover, one has $ \gamma(0, \phi_{\sigma}, {\rm Ad}, \psi) = \gamma(0, \phi_{\pi'}, {\rm Ad}, \psi)$ in view of \eqref{gam-red} and the equality $(\phi_\sigma, \tau_\sigma) = i_\theta^*(\phi_{\pi'}, \tau_{\pi'})$. Combining all the above, one obtains \eqref{E:main}.
\end{proof}

The above theorem is clearly compatible with the work of Ichino--Lapid--Mao \cite{ILM17}, Gan--Ichino \cite{GaIc14} and Gan--Savin \cite{GS12, GS12-2} on the formal degrees of discrete series of the metaplectic double cover $\Mp_{2r}$ of $\Sp_{2r}$. In particular, in  \cite{ILM17} the authors give a complete proof of the formal degree formula for the generic discrete series of $\Mp_{2r}$. It is natural to expect that the Hiraga--Ichino--Ikeda conjecture on formal degrees hold for arbitrary genuine discrete series of a cover $\wt{G}$.

\section{Whittaker dimensions of certain discrete series} \label{S:Wdim}
In this section, we consider the ``oasitic covers" of an almost simple simply-connected $G$ and concentrate on just the one-dimensional character $\sigma$ such that $\pi_\sigma$ is a discrete series, as studied by Borel in \cite{Bor76}. For such $\pi_\sigma \in \Irr_\epsilon(\wt{G})$, we compute  the $\psi_*$-Whittaker dimension
$$\dim \Wh_{\psi_*}(\pi_\sigma)$$
where $\psi_*: F \to \C^\times$ is of conductor $\mfr{p}_F$. The reason for this constraint on the conductor is that we will apply results from \cite{GGK} for the Gelfand--Graev module associated with $\psi_*$.

For the rest of the paper, we assume that $G$ is almost simple and simply-connected. Let $\wt{G}^{(n)}$ be the $n$-fold cover associated with $Q$ such that
$$Q(\alpha^\vee) = 1$$
for every short coroot $\alpha^\vee$ of $G$. We tabulate all oasitic covers $\wt{G}^{(n)}$ of such $G$ in the following table (see \cite[\S 6.1]{GGK}).

\begin{table}[H]  \label{T:oas}
\caption{Oasitic covers}
\vskip 10pt
\renewcommand{\arraystretch}{1.3}
\begin{tabular}{|c|c|c|c|c|c|c|c|c|c|c|}
\hline
 & $\SL_{r+1}$  &  $\Spin_{2r+1}$ & $\Sp_{2r}$  & $\Spin_{2r}$    \\
\hline
oasitic & $\gcd(n, r+1)=1$ & $n$ odd & $n$ odd & $n$ odd   \\ 
\hline
\end{tabular}
\vskip 10pt

\begin{tabular}{|c|c|c|c|c|c|c|c|c|c|c|}
\hline
 & $E_6$ & $E_7$& $E_8$ & $F_4$ & $G_2$   \\ 
\hline
oasitic & $2, 3\nmid n$ & $2, 3 \nmid n$& $2, 3, 5 \nmid n$ & $2, 3 \nmid n$ & $2, 3\nmid n$   \\  
\hline
\end{tabular}
\end{table}
\vskip 10pt

Several properties of oasitic covers are as follows:
\begin{enumerate}
\item[(i)] For every oasitic cover $\wt{G}^{(n)}$ as above, one has $Y_{Q,n} = nY$ and in particular, $\wt{G}^\vee \simeq G^\vee$ and $G_{Q,n} \simeq G$.
\item[(ii)] The additive character $\psi_*$ gives a character of the unipotent subgroup $U^-$ opposite to the unipotent radical $U$ of the Borel subgroup. Write $\mca{V}_\epsilon:={\rm ind}_{\pmb{\delta}_n \times U^-}^{\wt{G}} (\epsilon \times \psi_*)$ for the $\epsilon$-genuine Gelfand--Graev representation of $\wt{G}$ with left action given by $(g\cdot f)(x):=f(xg), f \in \mca{V}_\epsilon$. The structure of its Iwahori-component $\mca{V}_\epsilon^I$, as a module over the Iwahori-Hecke algebra $\HH_{\bepsilon}(\wt{G}, I)$, is given by (see \cite[Theorem 1.2]{GGK})
$$\mca{V}_\epsilon^I  \simeq \HH_{\bepsilon}(\wt{G}, I)  \otimes_{\HH_W}  V_\msc{X} .$$
Here $\msc{X}_{Q,n}:=Y/Y_{Q,n} = Y/nY$ is of size $n^r$ and affords a permutation representation 
$$\eta_\msc{X}: W \longrightarrow {\rm Perm}(\msc{X}_{Q,n})$$
given by $\eta_\msc{X}(w)(y):= w(y)$. Also, $V_\msc{X}$ is a deformation of $\eta_\msc{X} \otimes \varepsilon_W$ such that
\begin{equation} \label{E:dform}
(V_\msc{X})_{q\to 1} \simeq \eta_\msc{X} \otimes \varepsilon_W
\end{equation}
as $W$-representations, where $\varepsilon_W$ is the sign character of $W$ and the deformation $(-)_{q\to 1}$ is given as in \cite[Proposition 10.11.4]{Car}.
\end{enumerate}

For every $\epsilon$-genuine Iwahori-spherical representation $\pi$, we have 
$$\Wh_{\psi_*}(\pi):=\Hom_{\HH_{\epsilon}(\wt{G}, I)}(\mca{V}_{\bepsilon}^I, (\pi^\vee)^I).$$
In particular,
\begin{equation} \label{cal-Wd}
\Wh_{\psi_*}(\pi_\sigma) = \Hom_{\HH_W}(V_\msc{X}, \sigma^\vee|_{\HH_W}) = \Hom_{W}(\eta_\msc{X} \otimes \varepsilon_W, (\sigma^\vee|_{\HH_W})_{q\to 1}).
\end{equation}

For a one-dimensional character $\sigma: \HH_\epsilon(\wt{G}, I) \to \C$, Borel \cite[\S 5.8]{Bor76} determined explicitly the values of $\sigma$ on the Coxeter generators of $\HH_{\bepsilon}(\wt{G}, I) \simeq \HH(G_{Q,n}, I_{Q,n})$ such that $\pi_\sigma$ is a discrete series. For convenience, we set
$$
\mca{C}(\wt{G}):=\set{\sigma \in \Hom(\HH_{\bepsilon}(\wt{G}, I), \C): \ \pi_\sigma \text{ is a discrete series}}.$$

 Note that since we are considering oasitic cover of an almost simple $G$, the affine Weyl group of $G_{Q,n}$ is 
$$\wt{W}_{\rm ex} =(nY^{sc}) \rtimes W.$$
Let  $\wt{\Delta}^\vee$ be the set of extended simple coroots of $\wt{W}_{\rm ex}$. 
We set
$$\xi_{s}(\sigma) = 
\begin{cases}
1 &  \text{ if } \sigma(e_s^\theta) = q \text{ for } s\in \wt{\Delta}^\vee,\\
-1 &  \text{ if } \sigma(e_s^\theta) = -1 \text{ for } s\in \wt{\Delta}^\vee.
\end{cases}$$
It is clear that $\set{\xi_s(\sigma)}_{s\in \wt{\Delta}^\vee}$ uniquely determines the  character $\sigma$.

Let $\varepsilon_{\HH_I}$ be the sign character of $\HH_{\bepsilon}(\wt{G}, I)$ with $\xi_s(\varepsilon_{\HH_I}) =-1$ for all $s \in \wt{\Delta}^\vee$. Then $\varepsilon_{\HH_I} \in \mca{C}(\wt{G})$ and
$\pi_{\varepsilon_{\HH_I}}$ is the covering analogue of the Steinberg representation. In this case, it follows from \cite{GGK2} that
\begin{equation} \label{St-dim}
\dim \Wh_{\psi_*}(\pi_{\varepsilon_{\HH_I}}) =\val{W}^{-1} \cdot \prod_{j=1}^r (n + m_j),
\end{equation}
where $m_j, 1\lest j \lest r$ are the exponents of the Weyl group. 

When $G$ is of type $A, D$ or $E$, one has
$$\mca{C}(\wt{G}) = \set{\varepsilon_{\HH_I}}.$$
Thus it remains to consider $\sigma \in \mca{C}(\wt{G}) - \set{\varepsilon_{\HH_I}}$ when $G$ is of type $B, C, F_4$ and $G_2$.

\subsection{Type $B_r$}
For $\wt{\Spin}_{2r+1}^{(n)}, r\gest 2$, the extended Dynkin diagram for $\wt{W}_{\rm ex}$ is as follows:

$$ \qquad
\begin{picture}(4.5,0.3)(0,0)
\put(1,0){\circle*{0.08}}
\put(1.5,0){\circle{0.08}}
\put(2,0){\circle{0.08}}
\put(2.5,0){\circle{0.08}}
\put(3,0){\circle{0.08}}
%
\put(1.04,0.015){\line(1,0){0.42}}
\put(1.04,-0.015){\line(1,0){0.42}}

\multiput(1.55,0)(0.05,0){9}{\circle*{0.02}}
\put(2.04,0){\line(1,0){0.42}}
\put(2.54,0.015){\line(1,0){0.42}}
\put(2.54,-0.015){\line(1,0){0.42}}
%
\put(1,0.15){\footnotesize $\xi_3(\sigma)$}
\put(1.5,0.15){\footnotesize $\xi_1(\sigma)$}
\put(2,0.15){\footnotesize $\xi_1(\sigma)$}
\put(2.5,0.15){\footnotesize $\xi_1(\sigma)$}
\put(3,0.15){\footnotesize $\xi_2(\sigma)$} 
\end{picture}
$$
\vskip 10pt
\noindent Here the solid dot represents the unique simple coroot of $\wt{\Delta}^\vee$ not arising from $W$. It then follows from \cite[Page 254]{Bor76} that elements $\sigma \in \mca{C}(\wt{G}) - \set{\varepsilon_{\HH_I}}$ are of the form
$$ (\xi_1(\sigma), \xi_2(\sigma), \xi_3(\sigma)) =
\begin{cases}
(-1, -1, 1), (-1, 1, -1) \text{ or } (-1, 1, 1) & \text{ if $r\gest 4$}, \\
(-1, -1, 1), (-1, 1, -1) & \text{ if $r=2, 3$}.
\end{cases}
$$

If $\sigma$ has $(\xi_1(\sigma), \xi_2(\sigma)) = (-1, -1)$, then it follows from \eqref{cal-Wd} that
$$\dim \Wh_{\psi_*}(\pi_\sigma) = \dim \Wh_{\psi_*}(\pi_{\varepsilon_{\HH_I}})$$
 and thus is given explicitly by the same formula in \eqref{St-dim}. If $(\xi_1(\sigma), \xi_2(\sigma)) =(-1, 1)$, then in terms of the bipartition parametrization of $\Irr(W)$ (see \cite[\S 11.4]{Car}), we have 
$$(\sigma^\vee|_{\HH_W})_{q\to 1} = ( (1^r); \emptyset ).$$
Note that the character value of $\eta_\msc{X}$ was essentially computed in \cite{Som97, GNS99}. Moreover, it follows from \cite[Proposition 3.3]{GNS99} and \eqref{cal-Wd} that
$$\dim \Wh_{\psi_*}(\pi_\sigma) = \frac{1}{\val{W}} \cdot (n-1) \prod_{j=1}^{r-1} (n  + 2j - 1)$$
for these $\sigma \in \mca{C}(\wt{G})$ with $(\xi_1(\sigma), \xi_2(\sigma)) =(-1, 1)$.

\subsection{Type $C_r$} Consider $\wt{\Sp}_{2r}, r\gest 3$, the extended Dynkin diagram for $\wt{W}_{\rm ex}$ is as follows:

$$
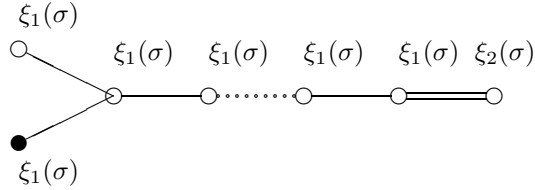
\begin{picture}(3.5,0.5)(0,0)
\put(0.5, 0.25){\circle{0.08}}
\put(0.5, -0.25){\circle*{0.08}}
\put(1,0){\circle{0.08}}
\put(1.5,0){\circle{0.08}}
\put(2,0){\circle{0.08}}
\put(2.5,0){\circle{0.08}}
\put(3,0){\circle{0.08}}
%
\put(1,0){\line(-2,1){0.46}}
\put(1,0){\line(-2,-1){0.46}}
\put(1.04,0){\line(1,0){0.42}}
\multiput(1.55,0)(0.05,0){9}{\circle{0.02}}
\put(2.04,0){\line(1,0){0.42}}

\put(2.54,0.015){\line(1,0){0.42}}
\put(2.54,-0.015){\line(1,0){0.42}}
\put(0.5,0.39){\footnotesize $\xi_1(\sigma)$}
\put(0.5,-0.44){\footnotesize $\xi_1(\sigma)$}
\put(1,0.19){\footnotesize $\xi_1(\sigma)$}
\put(1.5,0.19){\footnotesize $\xi_1(\sigma)$}
\put(2,0.19){\footnotesize $\xi_1(\sigma)$}
\put(2.5,0.19){\footnotesize $\xi_1(\sigma)$}
\put(2.9,0.19){\footnotesize $\xi_2(\sigma)$}

\end{picture}
$$
\vskip 35pt
\noindent
Again, it follows from \cite[Page 254]{Bor76} that $\mca{C}(\wt{G}) = \set{\sigma, \varepsilon_{\HH_I}}$ with
$$(\xi_1(\sigma), \xi_2(\sigma)) = (-1, 1).$$
In this case, the argument is exactly parallel to the type $B$ case and we have
$$\dim \Wh_{\psi_*}(\pi_\sigma) =\frac{1}{\val{W}} (n-1) \prod_{j=1}^{r-1} (n  + 2j - 1)$$
for this unique $\sigma$.

\subsection{Type $F_4$} 
The extended Dynkin diagram for the simple coroots $\wt{\Delta}^\vee$ of type $F_4$ is as follows:

$$
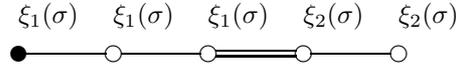
\begin{picture}(5.2,0.2)(0,0)
\put(1.5,0){\circle*{0.08}}
\put(2,0){\circle{0.08}}
\put(2.5,0){\circle{0.08}}
\put(3,0){\circle{0.08}}
\put(3.5,0){\circle{0.08}}
\put(1.54,0){\line(1,0){0.42}}
\put(2.04,0){\line(1,0){0.42}}
\put(2.54,0.015){\line(1,0){0.42}}
\put(2.54,-0.015){\line(1,0){0.42}}
\put(3.04,0){\line(1,0){0.42}}
\put(1.5,0.16){\footnotesize $\xi_1(\sigma)$}
\put(2,0.16){\footnotesize $\xi_1(\sigma)$}
\put(2.5,0.16){\footnotesize $\xi_1(\sigma)$}
\put(3,0.16){\footnotesize $\xi_2(\sigma)$}
\put(3.5,0.16){\footnotesize $\xi_2(\sigma)$}
\end{picture}
$$
\vskip 10pt
\noindent
We have $\mca{C}(\wt{G}) = \set{\sigma, \varepsilon_{\HH_I}}$ with
$$(\xi_1(\sigma), \xi_2(\sigma)) = (-1, 1).$$
In this case, 
$$(\sigma^\vee|_{\HH_W})_{q\to 1} = \phi_{1, 12}'$$
in the standard notation of parametrizing $\Irr(W)$, see \cite[Page 412]{Car}. It then follows from \cite[Page 26]{GNS99} that
$$\dim \Wh_{\psi_*}(\pi_\sigma) = \frac{1}{\val{W}} \cdot (n-1)(n-5)(n+1)(n+5)$$
for this $\sigma$, where $\val{W}=1152$.

\subsection{Type $G_2$}
We have the extended Dynkin diagram for $\wt{\Delta}^\vee$ of type $G_2$ as follows:
$$
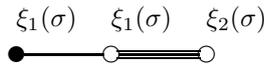
\begin{picture}(5.2,0.3)(0,0)
\put(2,0){\circle*{0.08}}
\put(2.5,0){\circle{0.08}}
\put(3,0){\circle{0.08}}
\put(2.04,0){\line(1,0){0.42}}
\put(2.53,0.018){\line(1,0){0.44}}
\put(2.54,0){\line(1,0){0.42}}
\put(2.53,-0.018){\line(1,0){0.44}}
%
\put(2,0.15){\footnotesize $\xi_1(\sigma)$}
\put(2.5,0.15){\footnotesize $\xi_1(\sigma)$}
\put(3,0.15){\footnotesize $\xi_2(\sigma)$}
\end{picture}
$$
\vskip 10pt
\noindent One has $\mca{C}(\wt{G}) = \set{\sigma, \varepsilon_{\HH_I}}$ with $\sigma$ determined by 
$$(\xi_1(\sigma), \xi_2(\sigma))=(-1, 1).$$
This gives that
$$(\sigma^\vee|_{\HH_W})_{q\to 1} = \phi_{1, 3}'$$
in the notation of \cite[Page 412]{Car}. It then follows from \cite[Page 26]{GNS99} that
$$\dim \Wh_{\psi_*}(\pi_\sigma) = \frac{1}{12} \cdot (n-1)(n+1)$$
for this $\sigma$.

%
%
%


\end{document}